 \documentclass[11pt]{article}
 \usepackage{amsmath,amsthm,amsfonts,amssymb}
\usepackage{verbatim}
\usepackage{latexsym}
\newtheorem{theorem}{Theorem}[section]
\newtheorem{lemma}[theorem]{Lemma}

\newtheorem{definition}[theorem]{Definition}
\newtheorem{corollary}[theorem]{Corollary}

\newtheorem{remark}[theorem]{Remark}

\newcommand\ZZ{\mathbb{ Z}}

\begin{document}
\title{A simple and consistent definition of homogeneous Besov spaces on stratified Lie groups}
\author{
Hartmut F\"uhr\\
\footnotesize\texttt{{fuehr@matha.rwth-aachen.de}}}

\maketitle

\begin{abstract}
We introduce a general definition of homogeneous Besov spaces on a stratified Lie group $G$, based on a Littlewood-Paley-type decomposition of Schwartz functions with all moments vanishing. We show that under mild and intuitive conditions the spaces thus defined are independent of the decomposition employed. A corollary of this is that previously constructed versions of homogeneous Besov spaces on $G$, relying on the spectral calculus of a sub-Laplacian of the group, are consistent, i.e., independent of the choice of sub-Laplacian.

 We further prove characterizations of homogeneous Besov spaces using continuous wavelet transforms, with a large variety of analysing wavelets to choose from. 
%


 \end{abstract}
%

\section{Introduction}\label{introduction}

In this note, we investigate Littlewood-Paley-decompositions and associated scales of Besov spaces on simply connected, connected stratified Lie groups. Previous definitions of homogeneous and non-homogeneous Besov and Sobolev spaces mostly relied on the spectral theory of a sub-Laplacian operator \cite{Folland75,Saka79,Furio-Melzi-Veneruso06,Ba,FuMa}. Briefly, a dyadic partition of unity of the positive half line is translated via the spectral calculus of the sub-Laplacian to yield a Littlewood-Paley-decomposition of functions and tempered distributions of the underlying group. Besov spaces are then defined in terms of ${\rm L}^p$-norms of the terms in the Littlewood-Paley decomposition. An alternative approach used the action of the heat semigroup associated to the sub-Laplacian, which may be interpreted as a continuous-scale Littlewood-Paley or continuous wavelet decomposition. While the connection between smoothness spaces and spectral theory is well established (see e.g. \cite{BuBe}), one definite drawback is that a priori, it is hard to tell to which extent these spaces depend on the choice of the sub-Laplacian. Since two such sub-Laplacians cannot be expected to commute, it seems very difficult to relate their spectral measures in any meaningful way. Accordingly, the existing proofs that Besov or Sobolev spaces associated to a sub-Laplacian are independent of the choice of sub-Laplacian (and thus, inherent to the group rather than to the operator), are quite subtle and technical, and they usually employ alternative characterizations of the spaces. See e.g. \cite{Folland75}. Also, many of the previous sources simply did not address the issue of independence of the choice of sub-Laplacian, at least not explicitly. For instance, in \cite{Saka79} independence can be concluded for certain inhomogeneous Besov spaces from \cite[Theorem 12]{Saka79} characterizing these spaces in terms of  moduli of continuity. 

In this paper, a more general approach is taken, which we now roughly sketch. We consider general Littlewood-Paley-decompositions of tempered distributions of the type
\[
  f = \lim_{M \to \infty} \sum_{|j| \le M} f \ast \psi_j^* \ast \eta_j~.
\] Here, $\psi_j,\eta_j$ are obtained from suitable $\psi,\eta \in \mathcal{S}(G)$ via dyadic dilations: $\psi_j = D_{2^j} \psi$, where $D_a$ denotes a dilation operator naturally associated to the stratified Lie group $G$. The convergence in the decomposition is required to hold in the Schwartz topology, for all Schwartz functions $f$ with all moments vanishing. Note that such decompositions are usually constructed in the above-mentioned way from a sub-Laplacian; in fact, it is the only general construction procedure for Littlewood-Paley-type decompositions that we are currently aware of. 

Then, for $1 \le p,q \le \infty$, the Besov semi-norms associated to $\psi$ are defined as 
\[
 \| u \|_{\dot{B}_{p,q}^{s,\psi}} = \left\| \left( 2^{js} \| u \ast \psi_j^* \|_p \right)_{j \in \mathbb{Z}} \right\|_{\ell^q(\mathbb{Z})}~.
\] We let $\dot{B}_{p,q}^{s,\psi}$ denote the space of all $u \in \mathcal{S}'(G)/\mathcal{P}$ (the quotient space modulo polynomials) for which this is finite. 
Our chief result, Theorem \ref{thm:main}, then states that, under very mild and intuitive conditions on the functions $\psi,\eta$,  the semi-norms and associated spaces do not depend on the Littlewood-Paley decomposition. As a corollary, these spaces coincide with the Besov spaces associated to sub-Laplacians, regardless of the precise choice of the sub-Laplacian. In particular, the latter are indeed  independent of the choice of sub-Laplacian. 

The final result, Theorem \ref{thm:main_2} considers continuous Littlewood-Paley or wavelet decompositions, showing e.g. for $q < \infty$, 
\[
\| u \|_{\dot{B}_{p,q}^s} \asymp \left( \int_{\mathbb{R}} a^{sq} \| f \ast D_a \psi \|_p^q \frac{da}{a} \right)^{1/q}~,
\]
Here, the restrictions on the analysing window $\psi$ can be further weakened. In particular, the admissible wavelets constructed in \cite{GeMa} are all available for the characterization of Besov spaces as well, provided they have sufficiently many vanishing moments. 

\section{Preliminaries and Notation}

$G$ always denotes a simply connected, connected stratified Lie group with Lie algebra $\mathfrak{g}$. Our source of reference for standard definitions and facts that are not explicitly mentioned in the following is \cite{FollandStein82}..
Since $G$ is exponential, we identify $G$ with $\mathfrak{g}$ as a set, and endow the latter with the Campbell-Hausdorff product; under this identification, the exponential map is just the identity map. Left and right Haar measure on $G$ is then given by Lebesgue measure. For a function $f$ on $G$ we define $\tilde{f}(x) = f(x^{-1})$ and $f^* = \overline{\tilde{f}}$. Since $G$ is unimodular, these operators preserve ${\rm L}^p$-norms. 
 
 The assumption that $G$ is stratified means that $\mathfrak{g}$ may be decomposed into a direct sum $\mathfrak{g} = \sum_{i=0}^m V_i$, with subspaces $V_i$ fulfilling $[V_i,V_j] \subset V_{i+j}$ and $[V_i,V_n] = \{ 0 \}$. The elements of the Lie algebra are canonically identified with left-invariant operators. We pick a basis of $\mathcal{g}$, obtained as a union of bases of the $V_i$. In particular, $Y_1,\ldots,Y_l$, for $l = \dim (V_1)$, is a basis of $V_1$. The {\em sub-Laplacian} associated to this basis is defined as 
\[
 \mathcal{L} = - \sum_{i=1}^l Y_i^2~,
\] first on $C_c^\infty(G)$, and then extended uniquely to an unbounded self-adjoint operator on ${\rm L}^2(G)$. 

Under our identification of $G$ with $\mathfrak{g}$, polynomials on $G$ are polynomials on $\mathfrak{g}$ (with respect to any linear coordinate system on the latter). Polynomials on $G$ are written as 
\begin{equation} \label{eqn:poly_gen_form} p\left(\sum_{i=1}^{\dim(G)}  x_i Y_i\right) =  \sum_I c_I x^I 
 \end{equation}
where $c_I \in \mathbb{C}$ are the coefficients, and $x^I=
x_1^{I_1}x_2^{I_2}\cdots x_n^{I_n}$ the monomials associated to the multi-indices $I \in \mathbb{N}^{\{ 1,\ldots,n\}}$. For a multi-index $I$, define
\[
 d(I) = \sum_{i=1}^n I_i n(i)~,n(i) = j \mbox{ for } Y_i \in V_j~.
\] A polynomial of the type (\ref{eqn:poly_gen_form}) is called {\em of homogeneous degree $k$} if $d(I) \le k$ holds, for all multiindices $I$ with $c_I \not= 0$. We write $\mathcal{P}_k$ for the space of polynomials of homogeneous degree $k$.

$\mathcal{S}(G)$ is the space of Schwartz functions on $G$; our identification entails 
$\mathcal{S}(G) = \mathcal{S}(\mathfrak{g})$.    $\mathcal{S}'(G)$ denotes the space of tempered distributions. We let 
\[
 \mathcal{S}(\mathbb{R}^+) = \{ f \in C^\infty(0,\infty) : \forall k , f^{(k)} \mbox{ decreases rapidly }, \lim_{\xi \to 0} f^{(k)} (\xi) \mbox{exists } \}  ~.\] We introduce the following convention for the spectral calculus of $\mathcal{L}$: Given any bounded Borel function $g$ on $\mathbb{R}^+$,  we write $g(\mathcal{L}) \delta$ for the convolution kernel associated to the operator $g(\mathcal{L})$. It is known that for $g \in \mathcal{S}(\mathbb{R}^+)$,  $g(\mathcal{L}) \delta$ is in $\mathcal{S}(G)$; see \cite{Hu} for details.

Dilations on $\mathfrak{g}$ are linear maps defined by 
\[
\delta_a \left( \sum_{i=1}^m X_i\right) = \sum_{i=1}^m a^i X_i~,~X_i \in V_i~.
\] They are Lie-algebra automorphisms of $\mathfrak{g}$, and consequently group automorphisms on $G$. 
The ${\rm L}^1$-normalized dilation operator on $G$, for any function $g$ on ${G}$ and $t>0$, is defined by
\begin{align}\notag
D_t g (x)= t^{Q}g(ax)~,
\end{align}
where $Q= \sum_1^m j(\dim V_j)$ is the {\em homogeneous dimension} of $G$.  
Observe that this action preserves the  $L^1$-norm, i.e., $\parallel
\phi_a\parallel_1= \parallel \phi\parallel$.  Furthermore, by straightforward calculation,
\begin{equation}
 D_s (D_t g) = D_{st} g~,~D_s(f \ast g) = (D_s f) \ast (D_s g)
\end{equation}

We fix a homogeneous
norm $|\cdot |$ on $G$ which is smooth away from $0$,  $|ax| =
a|x|$ for all $x \in G$, $a > 0$, $|x^{-1}| = |x|$ for all $x \in
G$, with $|x|
> 0$ if $x \neq 0$, and fulfilling a triangle inequality $|xy| \le C(|x|+|y|)$, with constant $C>0$. 

\section{Homogeneous Besov spaces defined via general Littlewood-Paley-decompositions}

Littlewood-Paley decompositions are based on building blocks with certain favourable properties. In this paper we focus on Schwartz functions, and the main additional property required are sufficiently many vanishing moments. 
\begin{definition}
 Let $N \in \mathbb{N}$. A function $f: G \to \mathbb{C}$ has decay order $N$ if  there exists a constant $C_f>0$ such that
\[
 |f(x)| \le C (1+|x|)^{-N}
\]
$f$ has vanishing moments of order $N$, if  one has 
\begin{equation}
 \label{eqn:defn_van_mom}\forall p \in \mathcal{P}_{N-1}~:~\int_G f(x) p(x) dx = 0~,
\end{equation} with absolute convergence of the integral. 
\end{definition}
Under our identification of $G$ with $\mathfrak{g}$, the inversion map $x \mapsto x^{-1}$ is identical to the additive inversion map. I.e., $x^{-1} = -x$, and it follows that $\tilde{p} \in \mathcal{P}_{N}$ for all $p \in \mathcal{P}_N$. Thus, if $f$ has vanishing moments of order $N$, then for all $p \in \mathcal{P}_{N-1}$ 
\[
  \int_G \tilde{f}(x) p(x) dx = \int_G f(x) \tilde{p}(x) dx  = 0,
\] i.e., $\tilde{f}$ has vanishing moments of order $N$ as well. 

We let $\mathcal{Z}(G)$ denote the space of Schwartz functions with all moments vanishing. The following properties of $\mathcal{Z}(G)$ will be used repeatedly (see \cite[Lemma 3.1]{FuMa}): $\mathcal{Z}(G)$ is a closed subspace of $\mathcal{S}(G)$, with $\mathcal{Z}(G) \ast \mathcal{S}(G) \subset \mathcal{Z}(G)$, and $f^* \in \mathcal{Z}(G)$ for all $f \in \mathcal{Z}(G)$. Furthermore, the dual of $\mathcal{Z}(G)$ is canonically identified with the quotient space $\mathcal{S}(G) / \mathcal{P}$. The topology of $\mathcal{S}'(G)/\mathcal{P}$ is just the topology of pointwise convergence on the elements of $\mathcal{Z}(G)$. Hence, for any net $(u_j + \mathcal{P} )_{j \in I}$, $u_j + \mathcal{P} \to u + \mathcal{P}$ holds if and only of $\langle u_j, \varphi \rangle = \langle u, \varphi \rangle$, for all $\mathcal{Z}(G)$. In the following, we will not explicitly distinguish between $u \in \mathcal{S}'(G)$ and $u+  \mathcal{P} \in \mathcal{S}'(G)/\mathcal{P}$, and write $u \in \mathcal{S}(G)/\mathcal{P}$ instead.

\begin{lemma} \label{lem:conv_SmP}
  For every $\psi \in \mathcal{S}(G)$, the map $u \mapsto u \ast \psi$ is well-defined and continuous on $\mathcal{S}(G)/\mathcal{P}$. If $\psi \in \mathcal{Z}(G)$, it is well-defined and continuous from $\mathcal{S}'(G)/\mathcal{P} \to \mathcal{S}'(G)$. 
\end{lemma}
\begin{proof}
Note that $\mathcal{P} \ast \mathcal{S}(G) \subset \mathcal{P}$. Hence $u \mapsto u \ast \psi$ induces a well-defined canonical map $\mathcal{S}'(G)/\mathcal{P} \to \left( \mathcal{S}'(G)/\mathcal{P} \right)^{\mathbb{Z}}$. Furthermore, $u \mapsto u \ast \psi$ is continuous on $\mathcal{S}'(G)$, as a consequence of \cite[Proposition 1.47]{FollandStein82}. Therefore, for any net $u_j \to u$ and any $\varphi \in \mathcal{Z}(G)$, the fact that $\varphi \ast \psi^* \in \mathcal{Z}(G)$ allows to write 
\begin{equation} \label{eqn:conv_cont} \langle u_j \ast \psi,\varphi \rangle = \langle u_j ,\varphi \ast \psi^* \rangle \to \langle u, \varphi \ast \psi^* \rangle = 
 \langle u \ast \psi, \varphi \rangle~,
\end{equation} showing $u_j \ast \psi \to u \ast \psi$ in $\mathcal{S}'(G)/\mathcal{P}$. 

For $\psi \in \mathcal{Z}(G)$, the fact that $\mathcal{P} \ast \psi = \{ 0 \}$ makes the mapping $u \mapsto u \ast \psi$ well-defined modulo polynomials. The continuity statement is proved by (\ref{eqn:conv_cont}), with assumptions on $\psi$ and $\varphi$ switched.  
\end{proof}

The definition of homogeneous Besov spaces requires taking ${\rm L}^p$-norms of elements of $\mathcal{S}(G)/\mathcal{P}$. The following remark clarifies this. 

\begin{remark} \label{rem:LP_norm_dist}
Throughout this paper, we use the canonical embedding  $\mathcal{L}^p(G) \subset \mathcal{S}'(G)$ canonically embedded. For $p< \infty$ this gives rise to an embedding $\mathcal{L}^p(G) \subset \mathcal{S}'(G)/\mathcal{P}$, using that $\mathcal{P} \cap {\rm L}^p(G) = \{ 0 \}$. Consequently, given $u \in \mathcal{S}'(G)/\mathcal{P}$, we let 
\begin{equation} \label{eqn:defn_lp_P} \| u \|_p = \| u + q \|_p \mbox{ whenever }u+q \in {\rm L}^p(G)~, \mbox{ for suitable } q \in \mathcal{P}
\end{equation} assigning the value $\infty$ otherwise. 
Here the fact that $\mathcal{P} \cap {\rm L}^p(G) = \{ 0 \}$ guarantees that the decomposition is unique, and thus (\ref{eqn:defn_lp_P}) well-defined.

By contrast, $\| \cdot \|_\infty$ can only be defined on $\mathcal{S}'(G)$, if we assign the value $\infty$ to $u \in \mathcal{S}'(G) \setminus {\rm L}^\infty(G)$. 

Note that with these definitions,  the Hausdorff-Young inequality $\| u \ast f \|_p \le \|u \|_p \| f \|_1$ remains valid for all $f \in \mathcal{S}(G)$, and all  $u \in \mathcal{S}'(G)/\mathcal{P}$ (for $p<\infty$), resp. $u \in \mathcal{S}'(G)$ (for $p=\infty$).
For $p=\infty$, this is clear. For $p< \infty$, note that if $u + q \in {\rm L}^p(G)$, then $(u+q) \ast \psi = u \ast \psi + q \ast \psi \in {\rm L}^p(G)$ with $q \ast \psi \in \mathcal{P}$. 
\end{remark}


%
We now introduce a general Littlewood-Paley-type decomposition. 
For this purpose, we define, given $\psi \in \mathcal{S}(G)$, 
\[
 \psi_j = D_{2^j} \psi ~.
\] 
\begin{definition} \label{defn:LP_S(G)}
 A pair of functions $(\psi,\eta) \in \mathcal{S}(G)^2$ is called LP-admissible if for all $f \in \mathcal{Z}(G)$, 
\begin{equation} \label{eqn:LP_schwartz}
 f = \lim_{N \to \infty} \sum_{|j| \le N} f \ast \psi_j^* \ast \eta_j
\end{equation} holds, with convergence in the Schwartz space topology.
 This entails the convergence
\begin{equation} \label{eqn:LP_dist}
 u = \lim_{N \to \infty} \sum_{|j| \le N} u \ast \psi_j^* \ast \eta_j~.
\end{equation} for all $u \in \mathcal{S}'(G)/\mathcal{P}$. 

A single function $\psi \in \mathcal{S}(G)$ is called LP-admissible, if there exists a suitable $\eta$ such that $(\psi,\eta)$ is $\mathcal{S}'(G)$-admissible. 
\end{definition}

Obviously, ``LP-admissible'' is shorthand for  ``Littlewood-Paley-admissible''. 

In an LP-admissible pair $(\psi,\eta)$, $\psi$ is called {\em analysis wavelet} and $\eta$ is called {\em synthesis wavelet}.

The existence of LP-admissible pairs is nontrivial. However, large classes can be constructed via the spectral calculus of sub-Laplacians:
\begin{lemma} \label{lem:construct_admissible_pair}
 There exist LP-admissible pairs $(\psi,\psi) \in \mathcal{Z}(G)^2$.
\end{lemma}
\begin{proof}
 Here we only sketch the construction. Pick $\widehat{\psi} \in C_c^\infty(\mathbb{R}^+)$ vanishing identically in a neighborhood of zero, and satisfying
\[
 1= \sum_{j\in \ZZ} |\hat\psi(2^{2j}\xi)|^2\quad a.e ~~ 
\]
Then the associated convolution kernel $\psi = \widehat{\psi}(\mathcal{L})\delta$ is in $\mathcal{Z}(G)$, and such that $(\psi,\psi)$ is LP-admissible \cite[Lemma 3.3]{FuMa}. 
\end{proof}

We note that the construction can be viewed as a straightforward generalization of the Littlewood-Paley-decompositions used e.g. in \cite{FrazierJawerth85}.

The following remark illustrates the additional freedom bought by the use of possibly distinct analysis and synthesis wavelets. It also shows that LP-admissible pairs need not be in $\mathcal{Z}(G)$. 
\begin{remark}
 Consider $G = \mathbb{R}$. Pick $\psi \in \mathcal{S}(\mathbb{R})$ such that $\widehat{\psi}$ is compactly supported, vanishes in a neighborhood of the identity, and fulfills 
\[
 \forall \psi \in \mathbb{R}\setminus{0}~:~1= \sum_{j\in \ZZ} |\hat\psi(2^{j}\xi)|^2 ~~ 
\] Note that throughout this remark, $\widehat{\cdot}$ denotes the usual Fourier transform on $\mathbb{R}$. Then it is not hard to see that $(\psi,\psi)$ is LP-admissible (this is essentially a special case of the above-cited \cite[Lemma 3.3]{FuMa}), and that $\psi \in \mathcal{Z}(G)$. Now pick $\varphi \in \mathcal{S}(\mathbb{R})$  with 
$\widehat{\varphi}$ supported in a neighborhood of $0$ where $\widehat{\psi}$ vanishes. Assume in addition that 
\[
 \widehat{\varphi} (0) \not=0~.
\] Then the convolution theorem yields that letting $\eta = \psi + \varphi$ implies $\eta^* \ast \psi = \psi^* \ast \psi$, and thus
$(\eta,\psi)$ is also LP-admissible, even though the zeroth moment of $\eta$ is nonzero. 
\end{remark}

An LP-admissible function $\psi$ gives rise to a Littlewood-Paley type decomposition of an element of $\mathcal{S}'(G)/\mathcal{P}$, which is the mapping $u \mapsto (u \ast \psi_j^*)_{j \in \mathbb{Z}}$. We now associate a scale of Besov-type spaces to the function $\psi$. 
\begin{definition} \label{defn:B_psi} 
 Let $(\psi,\eta) \in \mathcal{S}(G)^2$ be LP-admissible, let $1 \le p < \infty$, $1\le q \le  \infty$, and $s \in \mathbb{R}$. The {\bf homogeneous Besov space associated to $\psi$} is defined as 
\begin{equation}
 \label{eqn:def_B_psi}
\dot{B}_{p,q}^{s,\psi} = \left\{ u \in \mathcal{S}'(G)/\mathcal{P} : \left( 2^{js} \| u \ast \psi_j^* \|_p \right)_{j \in \mathbb{Z}} \in \ell^q(\mathbb{Z}) \right\}~,
\end{equation}
with associated norm 
\[
 \| u \|_{\dot{B}_{p,q}^{s,\psi}} = \left\|  \left( 2^{js} \| u \ast \psi_j^* \|_p \right)_{j \in \mathbb{Z}}  \right\|_{\ell^q(\mathbb{Z})}~.
\]  
For $p = \infty$, we define $\| \cdot \|_{\dot{B}_{p,q}^{s,\psi}}$ in an analogous way, but only for $(\psi,\eta) \in \mathcal{Z}(G)^2$.
\end{definition}

\begin{remark} The definition relies on the conventions regarding ${\rm L}^p$-norms of distributions (modulo polynomials), as outlined in Remark \ref{rem:LP_norm_dist}. For well-definedness in the case $p=\infty$, recall the observations in Lemma \ref{lem:conv_SmP}.
The Besov norm is a norm because of (\ref{eqn:LP_dist}). 
\end{remark}

\begin{remark}
This definition generalizes  homogeneous Besov spaces on $\mathbb{R}^n$, but also the constructions in \cite{Ba,FuMa}. 
\end{remark}

For all norm equivalences proved in this paper, the following lemma contains the crucial estimate. It formulates a central idea in wavelet analysis: In a convolution product of the type $g \ast D_t f$, vanishing moments of one factor together with smoothness of the other result in decay. Later on, we will apply the lemma to Schwartz functions $f,g$, where only the vanishing moment assumptions are nontrivial. The more general version given here is included for reference. 
\begin{lemma} \label{lem:decay_conv} Let $N,k \in \mathbb{N}$ be arbitrary. 
 \begin{enumerate}
\item[(a)]  Let $f \in C^{k}$, such that $Y^I(f)$ is of decay order $N$, for all $I$ with $d(I) \le k$. Let $g$ have vanishing moments of order $k$ and decay order $N+k +Q+1$. Then there exists a constant, depending only on $f$ and $g$, such that 
\begin{equation} \label{eqn:decay_conv1}
 \forall x \in G~ \forall~ 0 < t < 1 ~:~|g \ast (D_t f)(x)| \le C t^{k+Q} (1+|tx|)^{-N}~.
\end{equation}
 In particular, if $p>Q/N$,
 \begin{equation} \label{eqn:decay_conv1_norm}
 \forall x \in G~ \forall~ 0 < t < 1 ~:~\| g \ast (D_t f) \|_p \le C' t^{k+Q(1-1/p)} ~.
\end{equation}
 \item[(b)] Now suppose that $g \in C^k$, with $Y^I(\tilde{g})$ of  decay order $N$, and $f$ has vanishing moments of order $k$ and decay order $N+k+Q+1$. Then there exists a constant, depending only on $f$ and $g$, such that 
\begin{equation} \label{eqn:decay_conv2}
  \forall x \in G~ \forall~ 1 < t < \infty ~:~|g \ast (D_t f)(x)| \le C t^{-k}(1+|x|)^{-N} ~.
\end{equation}
 In particular, if $p>Q/N$, 
 \begin{equation} \label{eqn:decay_conv2_norm}
 \forall x \in G~ \forall~ 1 < t < \infty ~:~\|g \ast (D_t f)\|_p \le C' t^{-k} ~.
\end{equation}
\end{enumerate}
\end{lemma}
\begin{proof} 
First, let us prove $(a)$: 
For $x \in G$, let $P_{\tilde{f},x}^k$ denote the left Taylor polynomial of $\tilde{f}$ with homogeneous degree $k-1$, see \cite[1.44]{FollandStein82}. By that result, 
\[
 \left|f(y^{-1}x) - P^k_{x,f}(y)\right| \le C_k |y|^{k} \sup_{|z| \le b^{k} |y|,d(I) = k} |Y^I (D_t \tilde{f})(xz)|~,
\] with suitable positive constants $C_k$ and $b$. We next use the homogoneity properties of the partial derivatives\cite[p.21]{FollandStein82}, together with the decay condition on $Y^I f$ to estimate
\begin{eqnarray*}
\sup_{|z| \le b^{k} |y|}  |Y^I (D_t f)(xz)| & = & t^{k} \sup_{|z| \le b^k |y|} |D_t (Y^I f)(xz)| \\
 & = & t^{k+Q} \sup_{|z| \le b^k y} |(Y^I f)(t(x\cdot z))| \\
 & \le &  t^{k+Q} \sup_{z \le b^{k}|y|} C_f ( 1+|t(x \cdot z)|)^{-N} \\
 & \le &  t^{k+Q} \sup_{z \le b^{k}|y|} C_f (1+|tx|)^{-N} (1+|tz|)^N \\
 & \le &   t^{k+Q} (1+b)^{kN} C_f (1+|tx|)^{-N} (1+ |y|)^N~,
\end{eqnarray*} where the penultimate inequality used \cite[1.10]{FollandStein82}, and the final estimate used $|ty| = t |y| \le |y|$.  Thus
\[
  \left|f(xy) - P^k_{x,f}(y)\right| \le \tilde{C}_k  t^{k+Q} (1+|y|)^{N+k} (1+|tx|)^{-N}~.
\] Next, using vanishing moments of $g$, 
\begin{eqnarray*}
\lefteqn{\left| (g \ast D_t f) (x) \right|} \\ 
 & \le & \int_G |g(y)| \left| D_t f(y^{-1}x) -P_{x,D_t \tilde{f}}^k(y)\right| dy \\
 & \le & C (1+|tx|)^{-N}  t^{k+Q} \int_G |g(y)| ~ (1+|y|)^{N+k} dy \\
& \le & C (1+|tx|)^{-N}   t^{k+Q}  \int_G C_g (1+|y|)^{-Q-1} dy~,
\end{eqnarray*} and the integral is finite by \cite[1.15]{FollandStein82}. This proves (\ref{eqn:decay_conv1}), and 
(\ref{eqn:decay_conv1_norm}) follows by
\[
 \| g \ast D_t f\|_p \le C' t^{k+Q} \left( \int_G (1+|tx|)^{-Np} dx \right)^{1/p} \le C'' t^{k+Q-Q/p}  ~,
\] using $Np > Q$. 

For part (b), we first observe that 
\[
  (g \ast D_t f)(x) = t^Q \left(\tilde{f} \ast D_{t^{-1}} \tilde{g} \right) (t.x^{-1})~.
\] Our assumptions on $f,g$ allow to invoke part (a) with $\tilde{g},\tilde{f}$ replacing $f,g$, and (\ref{eqn:decay_conv2}) follows immediately. (\ref{eqn:decay_conv2_norm}) is obtained from this by straightforward integration.
\end{proof}

\begin{theorem}
 \label{thm:main}
 Let $(\psi^1,\eta^1),(\psi^2,\eta^2) \in \mathcal{S}(G)^2$ be LP-admissible pairs. Let $s \in \mathbb{R}$ and $1 \le q \le \infty$.
\begin{enumerate} 
 \item[(a)] Assume that $p< \infty$, and that all $\psi^i,\eta^i$ have vanishing moments of order $k>|s|$. Then  $\dot{B}_{p,q}^{s,\psi^1} = \dot{B}_{p,q}^{s,\psi^2}$, with equivalent norms. 
 \item[(b)] Assume that $\psi^i,\eta^i \in \mathcal{Z}(G)$. Then  $\dot{B}_{\infty,q}^{s,\psi^1} = \dot{B}_{\infty,q}^{s,\psi^2}$, with equivalent norms. 
\end{enumerate}
\end{theorem}
\begin{proof} The following proof works for $(a)$ and $(b)$ simultaneously. 
 It is sufficient to prove the norm equivalence, and here symmetry with respect to $\psi^1$ and $\psi^2$ immediately reduces the proof to showing, for a suitable constant $C>0$, 
\begin{equation} \label{eqn:half_norm_equiv}
  \forall u \in \mathcal{S}'(G)/\mathcal{P} : \| u \|_{\dot{B}_{p,q}^{s,\psi^1}} \le C \| u \|_{\dot{B}_{p,q}^{s,\psi^2}} ~,
\end{equation} in the extended sense that the left-hand side is finite whenever the right-hand side is. Hence assume that $u \in \dot{B}_{p,q}^{s,\psi^2}$; otherwise, there is nothing to show.  In the following, let $\psi_{i,j} = D_{2^j} \psi^i$, and $\eta_{i,j} = D_{2^j} \eta^i$, where $\eta^i$ is the reconstruction wavelet  associated to $\psi^i$ (i=1,2).

By LP-admissibility of $\psi^2$, 
\[
 u = \lim_{N \to \infty} \sum_{|j| \le N} u \ast \psi_{2,j}^* \ast \eta_{2,j}~, 
\] with convergence in $\mathcal{S}'(G)/\mathcal{P}$. 
Accordingly,
\begin{equation} \label{eqn:u_ast_eta}
 u \ast \psi_{1,\ell}^*= \lim_{N \to \infty} \sum_{|j| \le N} u \ast  \psi_{2,j}^* \ast  \eta_{2,j} \ast \psi_{1,\ell}^* ~,
\end{equation} where the convergence on the right-hand side holds in $\mathcal{S}'(G)/\mathcal{P}$, by \ref{lem:conv_SmP} (and even in $\mathcal{S}'(G)$, if $\psi_{1,\ell} \in \mathcal{Z}(G)$). We next show that the right-hand side also converges in ${\rm L}^p$. For this purpose, we observe that
\[
 \|  \eta_{2,j} \ast \psi_{1,\ell}^* \|_1 = \| D_{2^j} (\eta^2 \ast D_{2^{\ell-j}} \psi_1^{1*}) \|_1 = \| \eta^2 \ast D_{2^{\ell-j}} \psi^{1*} \|_1 
 \le C 2^{-|\ell- j|k}~.
\] For $\ell-j \ge 0$, this follows directly from (\ref{eqn:decay_conv2_norm}),  using $\psi^1,\eta^2 \in \mathcal{S}(G)$,  and vanishing moments of $\psi^1$, whereas for $\ell-j <0$, the vanishing moments of $\eta^2$ allow to apply (\ref{eqn:decay_conv1_norm}). 

Using Young's theorem, we estimate
\begin{eqnarray}
\nonumber \sum_{j \in \mathbb{Z}}  \| u \ast  \psi_{2,j}^* \ast \eta_{2,j} \ast \psi_\ell^* \|_p & \le &  \sum_{j \in \mathbb{Z}} \| u \ast \psi_{2,j}^*\|_p \| \psi_j \ast \eta_\ell^* \|_1 \\ & \le &   \| u \ast \psi_{2,j}^*\|_p  C_2 2^{-|j-\ell|k} \\
 & \le &  \label{eqn:est_sumj} C_2 \sum_{j  \in \mathbb{Z}} 2^{js} \| u \ast \psi_{2,j}^*\|_p   2^{-|j-\ell|k-js}~.
\end{eqnarray}
Next observe that 
\begin{equation} \label{eqn:matrix_decay} 2^{-|j-\ell|k - js}  = 2^{-\ell s} \cdot \left\{ \begin{array}{cc} 2^{-|j-\ell|(k+s) } & j \ge \ell \\ 2^{-|j-\ell| (k-s)} & j < \ell \end{array} \right. \le 2^{-\ell s} 2^{-|j-\ell|(k-|s|)}~.
\end{equation}
By assumption, the sequence $(2^{js} \| u \ast \psi_{j,2}^* \|_p)_{j \in \mathbb{Z}}$ is in $\ell^q$, in particular bounded. Therefore, $k-|s|>0$ yields that (\ref{eqn:est_sumj}) converges. But then the right-hand side of (\ref{eqn:u_ast_eta}) converges unconditionally with respect to $\| \cdot \|_p$. This limit coincides with the $\mathcal{S}'(G)/\mathcal{P}$-limit $u \ast \psi_{1,\ell}^*$ (which in the case of $\psi_{1,\ell}^* \in \mathcal{Z}(G)$ is even a $\mathcal{S}'(G)$-limit), yielding $u \ast \psi_{1,\ell}^* \in {\rm L}^p(G)$, with
\begin{eqnarray*}
 2^{\ell s} \| u \ast \psi_{1,\ell}^* \|_p & \le &  2^{\ell s} \sum_{j \in \mathbb{Z}}  \| u \ast  \psi_{2,j}^* \ast \eta_{2,j} \ast \psi_\ell^* \|_p \\
 & \le & C_3 2^{\ell s} \sum_{j \in \mathbb{Z}} 2^{js} \| u \ast \psi_{2,j}^* \|_p 2^{-|j-\ell| (k-|s|)}~. 
\end{eqnarray*}
Now an application of Young's inequality for convolution over $\mathbb{Z}$, again using $k - |s|>0$, provides (\ref{eqn:half_norm_equiv}). 
\end{proof}

As a consequence, we write $\dot{B}_{p,q}^s = \dot{B}_{p,q}^{s,\psi}$, for any LP-admissible $\psi \in \mathcal{Z}(G)$. These spaces coincide with the homogeneous Besov spaces of \cite{Ba,FuMa}, and with the usual definitions in the case $G=\mathbb{R}^n$. 
We now turn to decomposition with continuous scales. The following definition can be viewed as a continuous-scale analog of LP-admissibility.

\begin{definition}
 $\psi \in \mathcal{S}(G)$ is called {\bf admissible}, if for all $f \in \mathcal{Z}(G)$,
\[
  f = \lim_{\epsilon \to 0, A \to \infty} \int_\epsilon^A f \ast D_a(\psi^* \ast \psi) \frac{da}{a}~
\] holds with convergence in the Schwartz topology. 
\end{definition}

The next theorem reveals a large class of admissible wavelets. In fact, all the wavelets studied in \cite{GeMa} are also admissible in the sense considered here. Its proof is an adaptation of the argument showing \cite[Theorem 1]{GeMa}. 
\begin{theorem}
 Let $\widehat{h} \in \mathcal{S}(\mathbb{R}^+)$, and let $\psi = \mathcal{L} \widehat{h}(\mathcal{L}) \delta$.  Then $\psi$ is admissible up to normalization. 
\end{theorem}

\begin{proof}
The main idea of the proof is to write, for $f \in \mathcal{Z}(G)$, 
\begin{eqnarray*} \int_\epsilon^A f \ast D_a(\psi^* \ast \psi) \frac{da}{a} & = & f \ast \int_{\epsilon}^A D_a(\psi^* \ast \psi) \frac{da}{a} \\ & = & f \ast D_A g - f \ast D_\epsilon g~,\end{eqnarray*}
with $g \in \mathcal{S}(G)$. Once this is established, $f \ast D_A g \to c_g f$ for $A \to \infty$ follows by \cite[Proposition (1.49)]{FollandStein82}, with convergence in the Schwartz topology. Moreover, $f \in \mathcal{Z}(G)$ entails that $f \ast D_\epsilon g \to 0$: Given any $N>0$ and $I \in \mathbb{N}^n$ with associated left-invariant differential operator $Y^I$, we can employ (\ref{eqn:decay_conv1})  to estimate
\begin{eqnarray*}
 \sup_{x \in G} (1+|x|)^N  \left| (Y^I f \ast D_{\epsilon} g)(x) \right| & = & \sup_{x \in G} (1+|x|)^N \epsilon{Q+d(I)}
\left| f \ast D_{\epsilon} (Y^I g)(x) \right| \\
 & \le & C \sup_{x \in G} (1+|x|)^N \epsilon^{Q+d(I)+k}  (1+|\epsilon x|)^{-M} \\
 & \le & C \sup_{x \in G} (1+|x|)^{N-M} \epsilon^{Q+d(I)+k+M} ~,
\end{eqnarray*} which converges to zero for $\epsilon \to 0$, as soon as $M\ge N$ and $k > M-Q-d(I)$. But this implies $f \ast D_\epsilon g \to 0$ in $\mathcal{S}(G)$, by \cite{FollandStein82}. 

Thus it remains to construct $g$. To this end, define 
\begin{equation}
 \widehat{g}(\xi) = -\frac{1}{2} \int_{\xi}^\infty  a|\widehat{h}(a^2)|^2 da~,
\end{equation} which is clearly in $\mathcal{S}(\mathbb{R}^+)$, and let $g = \widehat{g}(\mathcal{L})\delta \in \mathcal{S}(G)$ denote the associated convolution kernel. 
Let $\varphi_1,\varphi_2$ be in $\mathcal{S}(G)$, and let $d\lambda_{\varphi_1,\varphi_2}$ denote the scalar-valued Borel measure associated to $\varphi_1,\varphi_2$ by the spectral measure. Then, by spectral calculus and the invariance properties of $da/a$, 
\begin{eqnarray*}
 \langle \int_{\epsilon}^A \varphi_1 \ast D_a (\psi^* \ast \psi) f \frac{da}{a}, \varphi_2 \rangle & = &  \int_{0}^\infty \int_\epsilon^A (a^2 \xi)^2 |h(a^2 \xi)|^2 \frac{da}{a} d\lambda_{\varphi_1,\varphi_2}(\xi)  \\
& = & \frac{1}{2}\int_{0}^\infty \int_{\epsilon^2 \xi}^{A^2 \xi} a |h(a^2 \xi)|^2 da d\lambda_{\varphi_1,\varphi_2}(\xi) \\
& = & \int_{0}^\infty \widehat{g}(A^2 \xi) - \widehat{g}(\epsilon^2 \xi)  d\lambda_{\varphi_1,\varphi_2}(\xi) \\
& = & \langle \varphi_1 \ast (D_A g - D_\epsilon g), \varphi_2 \rangle ~,
\end{eqnarray*} as desired. 
\end{proof}

Hence, by \cite[Corollary 1]{GeMa}:
\begin{corollary}
 \begin{enumerate}
  \item [(a)] There exist admissible $\psi \in \mathcal{Z}(G)$.
  \item[(b)] There exist admissible $\psi \in C_c^\infty(G)$ with vanishing moments of arbitrary order. 
 \end{enumerate}
\end{corollary}

Another popular candidate for defining scales of Besov spaces is the heat kernel and its derivatives. In our setting, the heat kernel of a sub-Laplacian corresponds to $\widehat{h}(\mathcal{L})$, with $\widehat{h}(\xi) = e^{-\xi}$, and the associated wavelets $\mathcal{L}^k h$ are generalizations of the well-known {\bf Mexican Hat} wavelet. Thus the following theorem also yields that the heat kernel definition of Besov spaces is equivalent to the initial one, thereby extending \cite[Theorem 3.10]{FuMa}. 

\begin{theorem} \label{thm:main_2}
Let $\psi \in \mathcal{S}(G)$ be admissible, with vanishing moments of order $k$. Then, for all $s \in \mathbb{R}$ with $|s|< k$, and all $1 \le p < \infty$, $1 \le q \le \infty$, the following norm equivalence holds: 
\begin{equation} \label{eqn:equiv_CWT_Besov}
\forall u \in \mathcal{S}'(G)/\mathcal{P}~:~ \| u \|_{\dot{B}_{p,q}^s} \asymp \left\| a \mapsto  a^{s} \| u \ast D_a \psi \|_p \right\|_{{\rm L}^p(\mathbb{R}^+; \frac{da}{a})}  \end{equation}
 Here the norm equivalence is understood in the extended sense that one side is finite iff the other side is. If $\psi \in \mathcal{Z}(G)$, it is also valid for the case $p=\infty$. 
\end{theorem}
\begin{proof}
The strategy consists in adapting the proof of Theorem \ref{thm:main} to the setting where one summation over scales is replaced by integration. This time however, we have to deal with both directions of the norm equivalence. 
In the following estimates, we use the symbol $\preceq$ to indicate an inequality that is true up to a fixed constant that may depend on $p$ and $q$, but is independent of $u$. 

Let us first assume that 
\[
 \int_{\mathbb{R}} a^{sq} \| f \ast D_a \psi \|_p^q \frac{da}{a} < \infty~,
\] for $u \in \mathcal{S}'(G)/\mathcal{P}$, $1 \le p,q \le \infty$, for an admissible function $\psi \in S(G)$ with $k_\psi>|s|$ vanishing moments ($\psi \in \mathcal{Z}(G)$, if $p=\infty$). Let $(\varphi,\varphi) \in \mathcal{Z}(G)$ denote an LP-admissible pair. Then, for all $j \in \mathbb{Z}$,  
\begin{equation} \label{eqn:decomp_u_varphi}
 u \ast \varphi_j ^*= \lim_{\epsilon \to 0, A \to \infty} \int_{\epsilon}^A u \ast D_a \psi^* \ast D_a \psi \ast \varphi_j^* \frac{da}{a}
\end{equation} holds in $\mathcal{S}'(G)$, by \ref{lem:conv_SmP}.

We next prove that the right-hand side of (\ref{eqn:decomp_u_varphi}) converges in ${\rm L}^p$. For this purpose, introduce
\[
 c_j = \int_{0}^\infty \| u \ast D_a \psi^* \ast D_a \psi \ast \varphi_j^* \|_p \frac{da}{a}~.
\]

 We estimate
\begin{eqnarray}
 \nonumber 
 c_j & \le & \int_{0}^\infty \| u \ast D_a \psi^* \|_p \| D_a \psi \ast \varphi_j^* \|_1 \frac{da}{a} \\
 & = & \label{eqn:anschluss_infty} \int_{1}^2 \sum_{\ell \in \mathbb{Z}} \| u \ast D_{a 2^\ell} \psi^* \|_p \| D_{a 2^\ell} \psi \ast \varphi_j^* \|_1 \frac{da}{a} \\
& \le & \label{eqn:anschluss_q} \left( \int_{1}^2 \left( \sum_{\ell \in \mathbb{Z}} \| u \ast D_{a 2^\ell} \psi^* \|_p \| D_{a 2^\ell} \psi \ast \varphi_j^* \|_1 \right)^q \frac{da}{a} \right)^{1/q} \log(2)^{1/q'}~, 
\end{eqnarray} where we used that $da/a$ is scaling invariant. Note that the last inequality is H\"older's inequality for $q< \infty$. In this case, taking $q$th powers and summing over $j$ yields
\begin{equation} \label{eqn:onedir_integrand}
 \sum_{j \in \mathbb{Z}} 2^{jsq} c_j^q \preceq
\int_1^2 \sum_{j \in \mathbb{Z}} 2^{jsq} \left(  \sum_{\ell \in \mathbb{Z}} \| u \ast D_{a 2^\ell} \psi^* \|_p \| D_{a 2^\ell} \psi \ast \varphi_j^* \|_1 \right)^q \frac{da}{a}~.
\end{equation}

 Using vanishing moments and Schwartz properties of $\psi$ and $\varphi$, we can now employ  (\ref{eqn:decay_conv1_norm}) and (\ref{eqn:decay_conv2_norm}) to obtain. 
\begin{equation} \label{eqn:l1_kernel}
 \| D_{a 2^\ell} \psi \ast \varphi_j^* \|_1 \preceq 2^{-|j-\ell|k}~,
\end{equation} with a constant independent of $a \in [1,2]$. But then, since $k>|s|$, we may proceed just as in the proof of \ref{thm:main} to estimate the integrand in (\ref{eqn:onedir_integrand}) via 
\begin{equation}
 \sum_{j \in \mathbb{Z}} 2^{jsq} \left(  \sum_{\ell \in \mathbb{Z}} \| u \ast D_{a 2^\ell} \psi^* \|_p \| D_{a 2^\ell} \psi \ast \varphi_j^* \|_1 \right)^q \preceq\sum_{\ell \in \mathbb{Z}} 2^{\ell s q} \| u \ast D_{a 2^\ell} \psi^* \|_p^q ~.
\end{equation}
Summarizing, we obtain 
\begin{eqnarray*}
 \sum_{j} 2^{jsq} c_j^q & \preceq &  \int_1^2 \sum_{\ell \in \mathbb{Z}} 2^{\ell s q} \| u \ast D_{a 2^\ell} \psi^* \|_p^q \frac{da}{a} \\
& \preceq & \int_{0}^\infty a^{sq}  \| u \ast D_{a 2^\ell} \psi^* \|_p^q \frac{da}{a} < \infty~.
\end{eqnarray*} In particular, $c_j < \infty$. But then the right-hand side of (\ref{eqn:decomp_u_varphi}) converges in ${\rm L}^p$, with limit $u \ast \varphi_j^*$. The Minkowski-inequality for integrals yields $\| u \ast \varphi_j^*\|_p \le c_j$, and thus 
\[
 \|u \|_{\dot{B}_{p,q}^s}^q \preceq  \int_{0}^\infty a^{sq}  \| u \ast D_{a 2^\ell} \psi^* \|_p^q \frac{da}{a}~,
\] as desired. In the case $q=\infty$,  (\ref{eqn:l1_kernel}) yields that 
\[
 \sup_{j}  2^{js} \left( \sum_{\ell \in \mathbb{Z}} \| u \ast D_{a 2^\ell} \psi^* \|_p \| D_{a 2^\ell} \psi \ast \varphi_j^* \|_1 \right) \preceq \sup_{\ell} 2^{\ell s} \| u \ast D_{a 2^\ell} \psi^* \|_p^q~.
\] Thus, by (\ref{eqn:anschluss_infty})
\begin{eqnarray*}
 \sup_{j} 2^{js} c_j & \preceq & \int_1^2 \sup_{\ell} 2^{\ell s} \| u \ast D_a 2^\ell \psi^* \|_p \frac{da}{a} \\
 & \preceq & {\rm ess ~sup}_{a} a^{s} \| u \ast D_a \psi^* \|_p~.
\end{eqnarray*} The remainder of the argument is the same as for the case $q < \infty$. 

Next assume $u \in \dot{B}_{p,q}^s$. Then, for all $a \in [1,2]$ and $\ell \in \mathbb{Z}$,
\[
 u \ast D_{a2^\ell} = \sum_{j \in \mathbb{Z}} u \ast \varphi_j^* \ast \varphi_j \ast D_{a2^\ell} \psi^*~, 
\]  with convergence in $\mathcal{S}'(G)/\mathcal{P}$; again, for $\psi \in \mathcal{Z}(G)$ convergence holds even in $\mathcal{S}'(G)$. As before,
\[
 \|\sum_{j \in \mathbb{Z}} u \ast \varphi_j^* \ast \varphi_j \ast D_{a2^\ell} \psi^* \|_p 
 \le \sum_{j \in \mathbb{Z}} \| u \ast \varphi_j^* \|_p \|  \varphi_j \ast D_{a2^\ell} \psi^* \|_1 ~.
\]
Again, we have $\| \varphi_j \ast D_{a2^\ell} \psi^* \|_1 \preceq 2^{-|j-\ell|k}$ with a constant independent of $a$. Hence one concludes in the same fashion as in the proof of Theorem \ref{thm:main} that, for all $a \in [1,2]$,  
\begin{equation}
 \left\| \left(  2^{\ell s} \| u \ast D_{a2^\ell} \psi^* \|_p \right)_{\ell \in \mathbb{Z}} \right\|_q \preceq \left\| \left( 2^{js} \| u \ast \varphi_j^* \|_p \right)_{j \in \mathbb{Z}} \right\|_q~,
\end{equation}
again with a constant independent of $a$. In the case $q=\infty$, this finishes the proof immediately, and for $q < \infty$, we 
integrate the $q$th power over $a \in [1,2]$ and sum over $\ell$ to obtain the desired inequality. 
\end{proof}

\begin{remark}
 Clearly, the proof of \ref{thm:main_2} can be accomodated to consider discrete Littlewood-Paley-decompositions based on integer powers of any $a >1$ instead of $a=2$. Thus consistently replacing powers of $2$ in Definitions \ref{defn:LP_S(G)} and \ref{defn:B_psi} by powers of $a > 1$ results in the same scale of Besov spaces. 
\end{remark}

\section*{Acknowledgements}
I thank Azita Mayeli for fruitful discussions.

\end{document}